\documentclass{elsart}
\usepackage{latexsym,amssymb}
\usepackage{flafter}
\usepackage{graphicx}
\usepackage{subfigure}
\usepackage{caption}
\usepackage[all]{xy}
\usepackage{float}

\usepackage{listings}
\newtheorem{theorem}{Theorem}[section]

\newtheorem{lemma}{Lemma}[section]

\newtheorem{corollary}{Corollary}[section]
\newtheorem{conjecture}{Conjecture}[section]
\newtheorem{conjectureexample}{Example}[section]
\newtheorem{proof}{Proof}[section]

\begin{document}
\begin{frontmatter}
\title{Divisor Goldbach Conjecture and its \\Partition Number}


\author[kun]{Kun Yan\corauthref{a}},
\author[biao]{Hou-Biao Li\corauthref{b}}
\corauth[a]{First author.}
\corauth[b]{Corresponding author.}
\ead{1939810907@qq.com,lihoubiao0189@163.com}
\address[kun]{School of Information and Software Engineering, University of Electronic Science and Technology of China, Chengdu, 610054, P. R. China}
\address[biao]{School of Mathematical Sciences, University of Electronic Science and Technology of
China, Chengdu, 610054, P. R. China}

\begin{abstract}
\quad Based on the Goldbach conjecture and arithmetic fundamental theorem, the Goldbach conjecture was extended to more general situations, i.e., any positive integer can be written as summation of some specific prime numbers, which depends on the divisible factor of this integer, that is:

For any positive integer $n~(n>2)$, if there exists an integer $m$, such that $m|n~( 1 < m < n )$, then $n=\sum_{i=1}^m  p_i $, where $ p_i~(i=1,2,3...m)$ is prime number.

In addition, for more prime summands,  the combinatorial counting is also discussed. For some special cases, some brief proofs are given.

By the use of computer, the preliminary numerical verification was given, there is no an anti-example to be found.
\end{abstract}


\end{frontmatter}


\section{Introduction}
In 1742, Goldbach wrote a famous letter to Euler, in which a conjecture on integer was presented. Euler expressed this conjecture as a more simplified form, i.e., any even number greater than 2 can be expressed as the sum of two primes. This is the well-known Goldbach conjecture\cite{ref1}. In 1855, Desboves verified Goldbach Conjecture for $n<10000$. The recent record is $n<4\times 10^{18}$\cite{verification}. No counter-example has been found to date. Till now, the best theory result is J. R. Chen's 1966 theorem that every sufficiently large integer is the sum of a prime and the product of at most two primes\cite{J.R.Chen}.

A pair of primes $(p,q)$ that sum to an even integer $E=p+q$ are known as a Goldbach partition(Oliveira e Silva). Letting $g(E)$, so-called Goldbach function, denote the number of Goldbach partitions of $E$ without regard to order. For example: \begin{center}100 = 3 + 97 = 11 + 89 = 17 + 83 = 29 + 71 = 41 + 59 = 47 + 53\end{center}
So $g(100)=6$. The Goldbach conjecture is then equivalent to the statement that $g(E)>0$ for every even integer $E>1$. Fig. \ref{Gcomet:fig} is a plot of $g(E)$ , which is obtained by Henry F. Fliegel \& Douglas S. Robertson in 1989\cite{FHR}. Hardy and Littlewood had derived an estimate for $g(E)$ in 1923\cite{hardy_Littlewood}:
\begin{equation}\label{hardy_littlewood_estimate}
g(E)\thicksim \frac{EC}{(\log \frac{E}{2})^2}\prod_{p>2,p|\frac{E}{2}}\frac{p-1}{p-2}.
\end{equation}
Here, $C=\prod_{p>2}(1-1/(p-1)^2)=0.66016......$\\
Some recent work for its lower/upper bounds could be found in \cite{FGAMF}\cite{MCW}. Obviously, it is very complicated.

\begin{figure}[H]
\centerline{\includegraphics[width=5in]{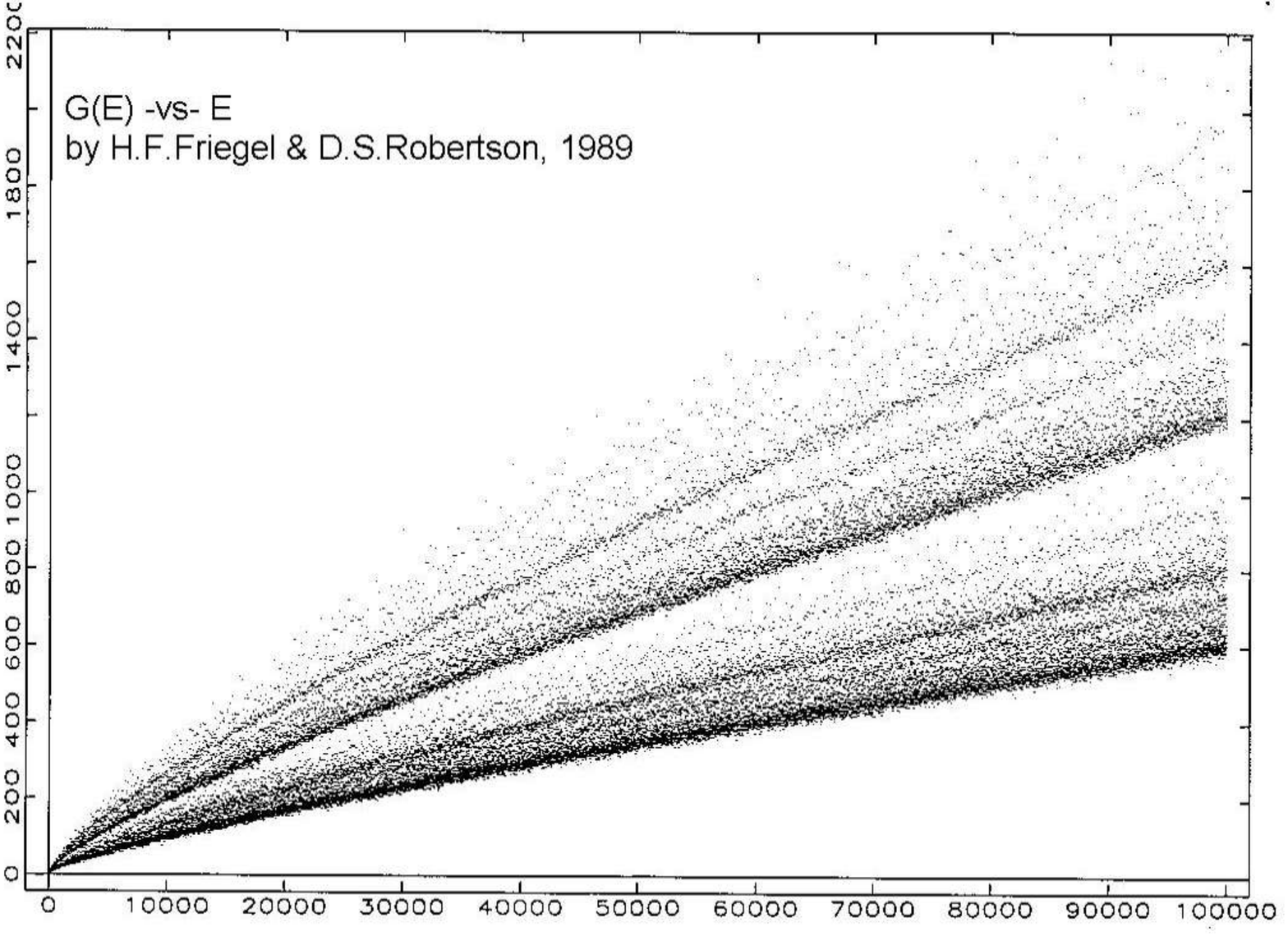}}
\caption{\small The Goldbach comet:the plot of $g(E)$ for $E$ up to 100000.}
\label{Gcomet:fig}
\end{figure}

There are a lot of research on integer splitting, one of the most famous is the work of the Hardy, Godfrey H. and John E. Littlewood. In 1923, they conjectured (as part of their famous Hardy-Littlewood prime tuple conjecture) that for any fixed $c>2$, the number of representations of a large integer $n$ as the sum of $c$ primes $n=p_1+...+p_c$ with $p_1\leq...\leq p_c$ should be asymptotically equal to$$(\prod_p\frac{p\gamma_{c,p(n)}}{(p-1)^c})\int_{2\leq x_1\leq ...\leq x_c:x_1+...+x_c=n}\frac{dx_1 ...dx_c}{\ln x_1 ...\ln x_c},$$ where the product is over all primes $p$, and $\gamma_{c,p}(n)$ is the number of solutions to the equation $n=(q_1+...+q_c)\bmod p$ in modular arithmetic, subject to the constraints $q_1,...,q_c \neq 0\bmod p$ \cite{hardy_Littlewood}. It was known as Hardy-littlewood conjecture.

Note that, the Goldbach conjecture is only for even number. Is there any special for even numbers? Obviously, for the other positive integers, the even number is special since it is divisible by 2. We also know that, in the era of Euclid, it has been proved that the fundamental theorem of arithmetic, the theorem divides any positive integer into two categories, one is prime number, and the other is non-prime, but all non-primes can be expressed as a form of products of some primes. So people hope to find the secret of all positive integer number by studying the prime number! In this paper, we will continue this problem and extend the Goldbach conjecture to a general form, i.e., any positive integer can be written as summation of some specific prime numbers, which depends on the divisible factor of this integer.

\section{Divisor Goldbach conjecture}
Traditionally, we divided a positive integer into odd and even numbers, but through the fundamental theorem of arithmetic we know that every positive integer has a single unique prime factorization\cite{RH}.
\begin{lemma}[Fundamental theorem of arithmetic\cite{ref2}]\label{lemma_Fundamental_Theorem_of_Arithmetic}
Every positive integer $n > 1 $ can be represented in exactly one way as a product of prime powers:
\begin{equation}\label{Fundamental_Theorem_of_Arithmetic}
n=p_1^{\alpha_1}p_2^{\alpha_2}p_3^{\alpha_3}...p_k^{\alpha_k}=\prod_i^kp_i^{\alpha_i},
\end{equation}
where $ p_1<p_2<...<p_k $ are primes and the $ \alpha_i $ are positive integers.
\end{lemma}

So, from the perspective of basic theorem of the arithmetic, the positive integer is simply divided into odd and even is very rough, since every even number has the factor 2 for its prime decomposition, and 2 is just a special case of many factors. Therefore, we can guess for other prime factors there may have similar splittings.
\begin{conjecture}\label{conjecture1}
For every positive integer $n > 1$ and $p$ is a prime factor of $n$,
then $ n $ can be expressed as the sum of $p$ primes , that is $$ n=\sum_{j=1}^{p} p_j, $$ where $ p_j~(j=1,2,...p)$ is prime number.
\end{conjecture}

\begin{conjectureexample}
~
\begin{enumerate}
\item $15=3\times 5$  ~has prime factor 3 and 5.

\begin{list}{}{}
\item 15=3+5+7 \hfill(the sum of 3 primes)
\item 15=2+2+3+3+5 \hfill(the sum of 5 primes)
\end{list}
\item $36=2^2\times 3^2$ ~has prime factor 2 and 3.
\begin{list}{}{}
\item 36=7+29 \hfill(the sum of 2 primes)
\item 36=2+11+23 \hfill(the sum of 3 primes)
\end{list}
\item $56=2^3\times 7$ ~has prime factor 2 and 7.
\begin{list}{}{}
\item 56=13+43 \hfill(the sum of 2 primes)
\item 56=2+3+3+5+7+13+23 \hfill(the sum of 7 primes)
\end{list}
\end{enumerate}
\end{conjectureexample}

Conjecture \ref{conjecture1} extends Goldbach's conjecture to any prime factor, what will happen if these prime factors are different or the same?
\begin{conjecture}\label{conjectureEnd}
For every positive integer $n > 2$ and if there exists a positive integer $m~(1<m<n)$, such that $m|n$, then $n$ can be expressed as the sum of $m$ primes, that is,
$$ n = \sum_{i=1}^m p_i,$$
where $p_i~(i=1,2,...,m)$ is primes.
\end{conjecture}
\begin{conjectureexample}\label{exampleEnd}
~
\begin{enumerate}
\item $36=2^3\times3^2$, ~$4|36$ and $6|36$.
\begin{list}{}{}
\item 36=3+5+11+17 \hfill(the sum of 4 primes)
\item 36=2+2+5+7+7+13 \hfill(the sum of $6$ primes)
\end{list}
\item $56=2^3\times 7$, ~$8|56$ and $14|56$.
\begin{list}{}{}
\item 56=2+2+5+5+7+11+11+13 \hfill (the sum of 8 primes)
\item 56=2+2+2+2+2+2+2+2+3+3+5+7+11+11 (the sum of $14$ primes)
\end{list}
\end{enumerate}
\end{conjectureexample}

Obviously this conjecture can be seen as a generalized form of Goldbach Conjecture, therefore we named it \textbf{divisor Goldbach conjcture}. The proof of this conjecture may be difficult, we mainly give some proofs in some special cases and some empirical verification by using computers. The corresponding program code has been put in Appendix I (C language), the numerical results can also be found in Appendix II.

As mentioned in the  previous section, the hardy-littlewood conjecture \cite{hardy_Littlewood} can also be seen as the goldbach conjecture 's promotion. But the above divisor Goldbach conjecture differs the hardy-littlewood conjecture since the divisor Goldbach conjecture does not require the integer large enough and its splitting number is just its divisor, which is more similar to Goldbach conjecture. However, its splitting form is not unique for any integer, which is the same as hardy-littlewood conjecture. Next, we will continue researching this problem.

\section{On the combinatorial counting problem}
As we can see from the Example \ref{exampleEnd}, $4|36$ and $36=3+5+11+17$, we denote primes sequence $(3,5,11,17)_4$ be a \textbf{divisor Goldbach partition} of $(4,36)$ without regard to order. Obviously, the divisor Goldbach partition is not unique. By using computer we know (Appendix \uppercase\expandafter{\romannumeral 2})
\begin{center}
  36=2+2+3+29=2+2+13+19=3+3+7+23=3+3+11+19=3+3+13+17=...
\end{center}
So, for any integer number $m$ and $n$, if $m|n$ and $n=p_1+...+p_m$, where $p_i(i=1,2,...,m)$ is prime with $p_1\leq...\leq p_m $, then we denote the primes sequence $(p_1,...,p_m)_m$ be a divisor Goldbach partition of $(m,n)$, and all the divisor Goldbach partitions of $(m,n)$ consist of the following set $S_{m|n}$,
\begin{equation}
S_{m|n}=\{(p_1,p_2,...,p_m)_m|\sum_{i=1}^m p_i = n, p_1\leq p_2\leq ...\leq p_m,p_i~is~prime\}.
\end{equation}
Next, we define a function $Y(m,n)$ which means the number of divisor Goldbach partitions of $(m,n)$ with $m|n$. Obviously,
\begin{equation}
Y(m,n)=Card (S_{m|n}).
\end{equation}
The divisor Goldbach conjecture is then equivalent to the statement that $Y(m,n)>0$ for every positive integer pair $(m,n)$ with $m|n$ and $1<m<n$. So, $Y(4,36)=15$ since 36 can be expressed as the sum of 4 primes in 15 different ways(see Appendix \uppercase\expandafter{\romannumeral 2}). By the way, if $m\equiv 2$ in function $Y(m,n)$, it was just the Goldbach function
\begin{equation}\label{g(E)}
g(E)=Y(2,n),~ (2|n).
\end{equation}
By the way, if the condition $m|n$ is not necessary in function $Y(m,n)$, the function $Y(m,n)$ asymptotically equal to\cite{hardy_Littlewood}
$$(\prod_p\frac{p\gamma_{m,p(n)}}{(p-1)^m})\int_{2\leq x_1\leq ...\leq x_m:x_1+...+x_m=n}\frac{dx_1 ...dx_m}{\ln x_1 ...\ln x_m},$$
however, it requires the positive integer $n$ is sufficiently large. In this paper, we mainly study the function $Y(m,n)$ under the condition $m|n$ with any integer $n$. Next, we proved a part of divisor Goldbach conjecture and got some algebraic ralationship of function $Y(m,n)$.
\begin{theorem}\label{T_p_is_2}
For any positive integer $n > 1$ and if there exists a positive integer $m~(1<m<n)$, such that $n/m=2$, then $n$ can be expressed as the sum of $m$ primes and $Y(m,2m)\equiv 1$.
\end{theorem}
\begin{proof}
If $n/m=2$, obviously the set $S_{m|n}$ has an element $s=(p_1,p_2,...,p_m)$ with $p_i~(i=1,2,...,m)=2$. So, in this case the conjecture \ref{conjectureEnd} is correct. Next, we prove $Y(m,2m)\equiv 1$. Since $ Y(m,2m)=Card(S_{m|2m})$, we have to prove $Card(S_{m|2m}) \equiv 1$,
equivalently, the set $S_{m|2m}$ has only one element $s=(2,2,...,2)_m$.

Now, we prove this problem by contradiction. Assume $Card(S_{m|2\cdot m})>1$, since $s=(2,2,...,2)_m\in S_{m|2\cdot m}$,
so there exists another element $s_1=(p_1,p_2,...,p_m)_m\in S_{m|2m}~~ and~~ s\neq s_1.$ Note that $ \sum(s_1) = 2m$, since $(\sum s_1)/m=(\sum_{i=1}^m p_i)/m=2.$ Therefore, in the element $s_1$, there must be some primes larger than 2,and other primes less than 2,
since 2 is the least prime and there do not exist primes that are less than 2, therefore $$s_1 \notin S_{m|2m}.$$
So there does not exist an element $s_1=(p_1,p_2,...,p_m)_m\in S_{m|2m}$ with $s\neq s_1$.
So, the assumption does not hold. Note that $ Card(S_{m|2m})\equiv 1$. So
\begin{equation}\label{2m=1}Y(m,2m)\equiv 1.\end{equation}
\qed
\end{proof}

\begin{theorem}\label{T_p_is_p}
For any positive integer $n > 1$ and if there exists a positive integer $m~(1<m<n)$, such that $n/m=p$, where $p$ is a prime number, then $n$ can be expressed as the sum of $m$ primes. And $Y(m,p\cdot m)\geq 1$.
\end{theorem}
\begin{proof}
Since $n/m=p$, and $p$ is prime number, let $s=(p,p,...,p)_m,$ obviously $\sum s = m\cdot p=n.$ Note that the element $$s=(p,p,...,p)_m\in S_{m|p\cdot m}.$$
So, the set $S_{m|p\cdot m}$ at least have one element. So \begin{equation}\label{pmgeq}Y(m,p\cdot m)\geq 1.\end{equation}
\qed
\end{proof}

We have proved Theorem \ref{T_p_is_p}, when $n/m=p$, where $p$ is prime number. Next, for convenience, we denote $S^{p(z)}_{m|n}$ as a subset of $S_{m|n}$, where\begin{equation}\label{Sp(z)}S^{p(z)}_{m|n}=\{(p_1,...,p_{k+1},...,p_{k+z},...,p_m)_m|p_k<p_{k+1}=p_{k+z}=p<p_{k+z+1}\}.\end{equation}
Obviously, \begin{equation}\label{cap_and_cup}S^{p(z_1)}_{m|n}\cap S^{p(z_2)}_{m|n}=\emptyset~(z_1\neq z_2),~and~\cup_{z=0}^m S^{p(z)}_{m|n}=S_{m|n}.\end{equation}
Let $A$ be a finite set and its element is integer, and the maximum element of set A, $Max(A)\leq m$, then $S^{p(A)}_{m|n}=\cup_{z\in A}S^{p(z)}_{m|n}$, also we denote
\begin{equation}Y_{p(z)}(m,m\cdot p)=Card(S^{p(z)}_{m|n}),~and~Y_{p(A)}(m,m\cdot p)=Card(S^{p(A)}_{m|n}).\end{equation} Obviously, we can derive \begin{equation}Y_{p(m)}\label{Yp_special}(m,m\cdot p)=1,~and~Y_{p(m-1)}(m,m\cdot p)=0~(m>1).\end{equation}
Clearly, \begin{equation}\label{Y=sum1}Y(m,m\cdot p)=\sum_{i=0}^{m}Y_{p(i)}(m,m\cdot p).\end{equation}
When $n/m=p$ and $(n+p)/(m+1)=p$, where $p$ is prime number, we can got the set $S_{m|n}$ and the set $S_{(m+1)|(n+p)}$, and define a function $f$
$$f:S_{m|n}\longrightarrow S_{(m+1)|(n+p)} - S_{(m+1)|(n+p)}^{p(0)},$$
Let $s_m\in S_{m|n}^{p(z)}$, $s_{m+1}\in S_{(m+1)|(m+1)\cdot p}^{p(z+1)}$ with \\
\begin{scriptsize}
\xymatrix{
s_m = &(p_1,\ar[d]^=&p_2,\ar[d]^=&...p_{k+1},&...p_{k+z},&...p_m\ar[dr]^=&)_m\\
s_{m+1} = &(p_1,&p_2,&...p_{k+1},&...p_{k+z},&p_{k+z+1},&...p_{m+1}&)_{m+1}
}
\end{scriptsize}\\
the function $f$ is defined as $$f: s_m \longmapsto s_{m+1} = f(s_m).$$
So, for any element $s\in S_{m|n}^{p(z)}$, we can easy derive $f(s)\in S_{(m+1)|(m+1)\cdot p}^{p(z+1)}$, note that
\begin{equation}\label{m_subset_m1}f(S_{m|m\cdot p}^{p(z)})\subset S^{p(z+1)}_{(m+1)|(m+1)\cdot p}.\end{equation}
So, we can easy derive that the codomain of $f$ \begin{equation}\label{Z_f}Z_f\subset S_{(m+1)|(m+1)\cdot p}-S_{(m+1)|(m+1)\cdot p}^{p(0)}.\end{equation}
On the other hand, for any element $s, s'\in S_{m|n}$, and $s\neq s'$, clearly $f(s)\neq f(s')$.
So, for function $f$, its inverse function $f^{-1}$ exists when $D_{f^{-1}}=Z_f$. As mentioned in above, $s_{m+1}=f(s_m)$, so, $s_m=f^{-1}(s_{m+1})$. And we can easily derive that for any element $s_{m+1}\in S_{(m+1)|(m+1)\cdot p}^{p(z+1)}~(z\geq 0)$, $\sum{(f^{-1}(s_{m+1}))}=m\cdot p$, note that
\begin{equation}\label{m1_subset_m}f^{-1}(S_{(m+1)|(m+1)\cdot p}^{p(z+1)})\subset S_{m|m\cdot p}^{p(z)}.\end{equation}
So, according to Eq. (\ref{m_subset_m1}) and Eq. (\ref{m1_subset_m}), we can get
\begin{equation}\label{f(S_m)=S(m+1)}f(S_{m|m\cdot p}^{p(z)}) = S^{p(z+1)}_{(m+1)|(m+1)\cdot p},~and~f^{-1}(S_{(m+1)|(m+1)\cdot p}^{p(z+1)})= S_{m|m\cdot p}^{p(z)}.\end{equation}
\begin{theorem}\label{1=2}
For any positive integer $n > 1$ and if there exists a positive integer $m~(1<m<n)$, such that $\frac{n}{m}=p$, and $p$ is prime, then $n$ can be expressed as the sum of $m$ primes. And the function $Y(m,n)$ satisfies
\begin{equation}Y_{p(z)}(m,m\cdot p)=Y_{p(z+1)}(m+1,(m+1)\cdot p),~(z\geq 0, z\in Z).\end{equation}
\end{theorem}
\begin{proof}
When $m|m\cdot p$ and $(m+1)|(m+1)\cdot p$, for the set $S^{p(z_0)}_{m|m\cdot p}$ and the set $S^{p(z_0+1)}_{(m+1)|(m+1)\cdot p}$, according to equation (\ref{f(S_m)=S(m+1)}), we can easily derive that
\begin{equation}Card(S^{p(z_0)}_{m|m\cdot p})=Card(S^{p(z_0+1)}_{(m+1)|(m+1)\cdot p}),\end{equation}
so
$$Y_{p(z)}(m,m\cdot p)=Y_{p(z+1)}(m+1,(m+1)\cdot p).$$
\qed
\end{proof}

According to Theorem \ref{1=2}, equation (\ref{Y=sum1}) and equation (\ref{Yp_special}) we can get a diagram as belows:\\~\\
\begin{scriptsize}
\xymatrix{
Y(2,2p)=Y_{p(0)}(2,2p)\ar[dr]^{=}&+Y_{p(2)}(2,2p)\ar[dr]^{=}\\
Y(3,3p)=Y_{p(0)}(3,3p)\ar[dr]^{=}&+Y_{p(1)}(3,3p)\ar[dr]^{=}&+Y_{p(3)}(3,3p)\ar[dr]^{=}\\
Y(4,4p)=Y_{p(0)}(4,4p)\ar[dr]^{=}&+Y_{p(1)}(4,4p)\ar[dr]^{=}&+Y_{p(2)}(4,4p)\ar[dr]^{=}&+Y_{p(4)}(4,4p)\ar[dr]^{=}\\
Y(5,5p)=Y_{p(0)}(5,5p)&+Y_{p(1)}(5,5p)&+Y_{p(2)}(5,5p)&+Y_{p(3)}(5,5p)&+Y_{p(5)}(5,5p)\\
 &\textbf{......}
 }
\end{scriptsize}\\
Adding the results we have got from the diagram, we can easily derive the result as bellow:

\begin{corollary}\label{G_m(mp)}
For any positive integer $n > 1$ and if there exists a positive integer $m~(1<m<n)$, such that $\frac{n}{m}=p$, and $p$ is prime, then $n$ can be expressed as the sum of $m$ primes. And
\begin{equation}Y(m,m\cdot p)=\sum_{j=2}^{m}Y_{p(0)}(j,j\cdot p)+1.\end{equation}
\end{corollary}

\section{Some empirical verification results of $Y(m,n)$}
In the previous section, we present divisor Goldbach conjecture as a generalized form of Goldbach's conjecture, and reseached the conjecture in a few special cases. Next, we use computer to discover some very interesting rules on function $Y(m,n)$.

In Fig. \ref{m_d_n}, n -axis means integer number $n>2$, and m -axis means divisor of $n$, the point means $m|n$, all points consists of a set $D$ (the figure only shows the part result of this set), which is the domain of the function $Y(m,n)$,
\begin{equation}D=\{(m,n)|m\in N,n\in N,n>1,m>1,m|n\}.\end{equation}
By the way, the Fig. \ref{m_d_n} is a kind of Sieve of Eratosthenes\cite{The_Sieve_of_Eratosthenes}. Because the prime set $P$ satisfies
\begin{equation}P=\{n|Not~exists~integer~number~m>1~let~(m,n)\in D, n\in N\}.\end{equation}
As we can see from Fig. \ref{m_d_n}, each point on a line. Note that
\begin{equation}n=k\cdot m ~ (k=2,3,...).\end{equation}
When $k=2$, we have proved $Y(m,2m)\equiv 1$ (see Eq. (\ref{2m=1})). And proved $Y(m,k\cdot m)\geq 1$ (se Eq. (\ref{pmgeq})) when $k$ is prime.

\begin{figure}[H]
\centerline{\includegraphics[width=7in]{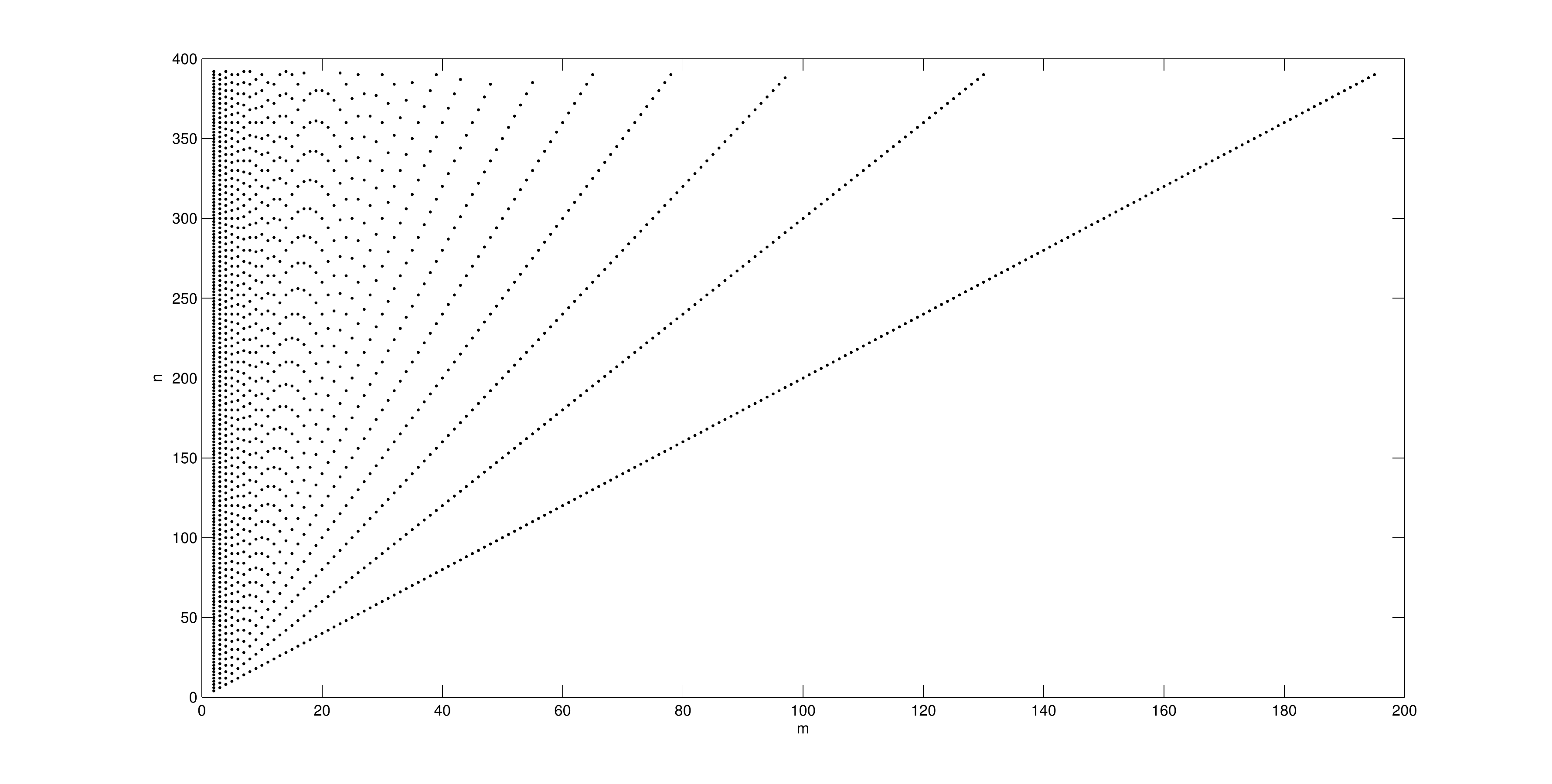}}
  \begin{picture}(0,0)
    \put(200,134){$k=2$}
  \end{picture}
    \begin{picture}(0,0)
    \put(176,167){$k=3$}
  \end{picture}
    \begin{picture}(0,0)
    \put(157,196){$k=4$}
  \end{picture}
\caption{\small The set $D$: domain of function $Y(m,n)$.}
\label{m_d_n}
\end{figure}
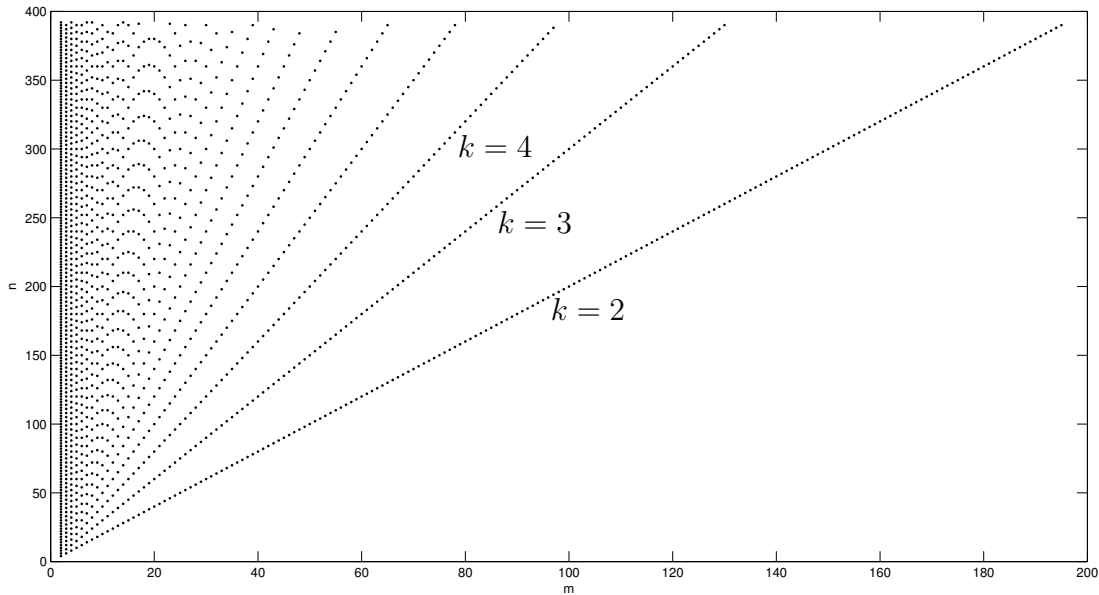
 The Fig. \ref{num} has plotted the results of $Y(m,n)$ by some numerical experimentations. According to Eq. (\ref{g(E)}), when $m\equiv 2$, the function $Y(2,2m)$ is Goldbach function $g(E)$ with $E$ is even number. Note that in the Fig. \ref{num} fixed $m=2$ then we can get the plot of function $g(2m)$, so-called Goldbach function, and it was shown as Fig. \ref{Gcomet:fig}. The Fig. \ref{diff_m} shows the value of $Y(m,n)$ for different $m$ values, as we know, the first subfigure is Goldbach comet (see Fig. \ref{Gcomet:fig}).  However, it looks very complicated. In addition, as we can see from Fig. \ref{num} , when we fixed the quotient of $n/m$ (let $k\geq 2$), then the $Y(m,m\cdot k)$ may has a good regular pattern with argument $m$. The Fig. \ref{diff_k} shows the function $Y(m,k\cdot m)$.

 \begin{figure}[H]
\centerline{\includegraphics[width=7in]{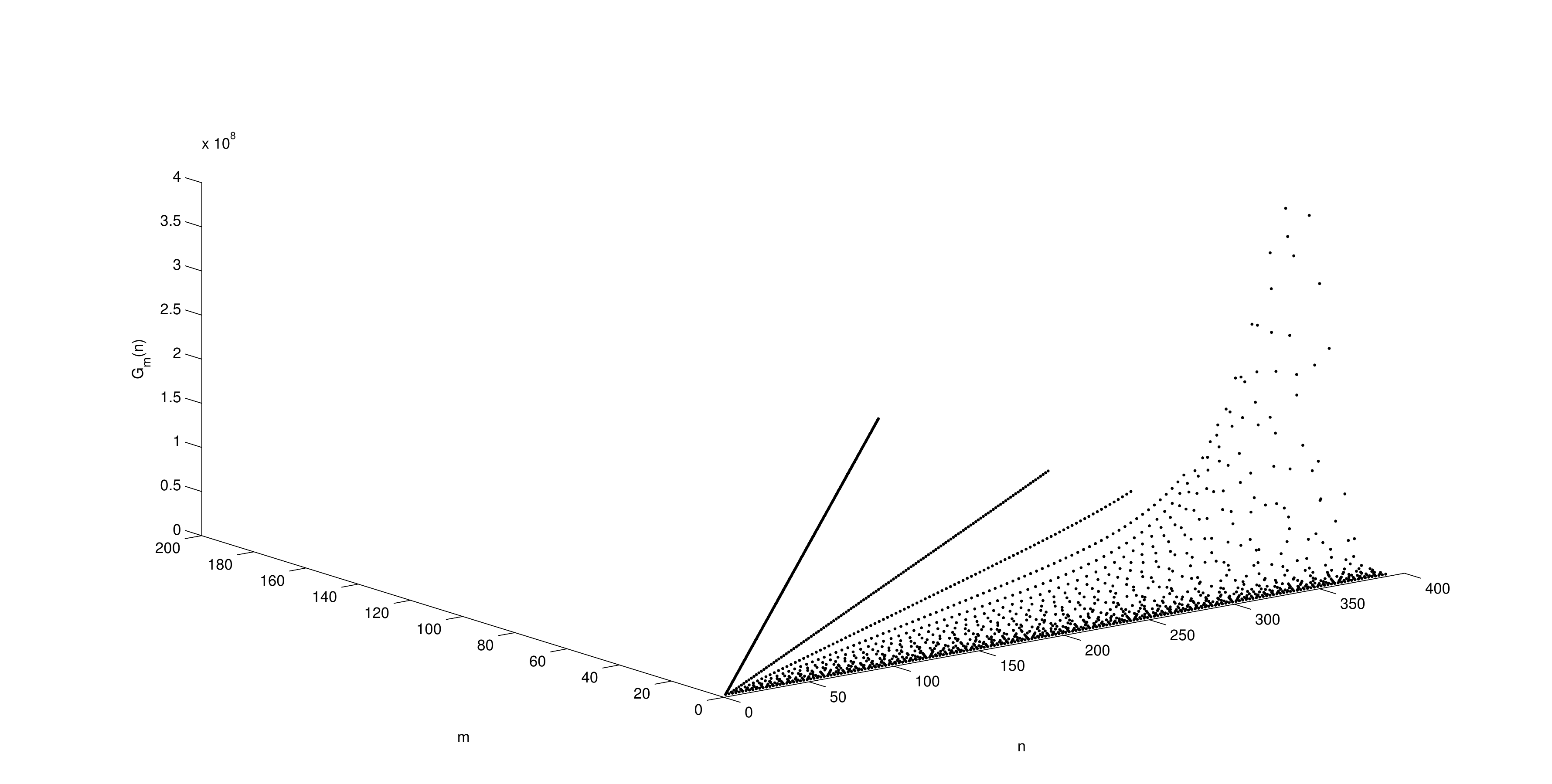}}
  \begin{picture}(0,0)
    \put(15,200){$Y(m,n)$}
  \end{picture}
\caption{\small The function $Y(m,n)$}\label{num}
\end{figure}
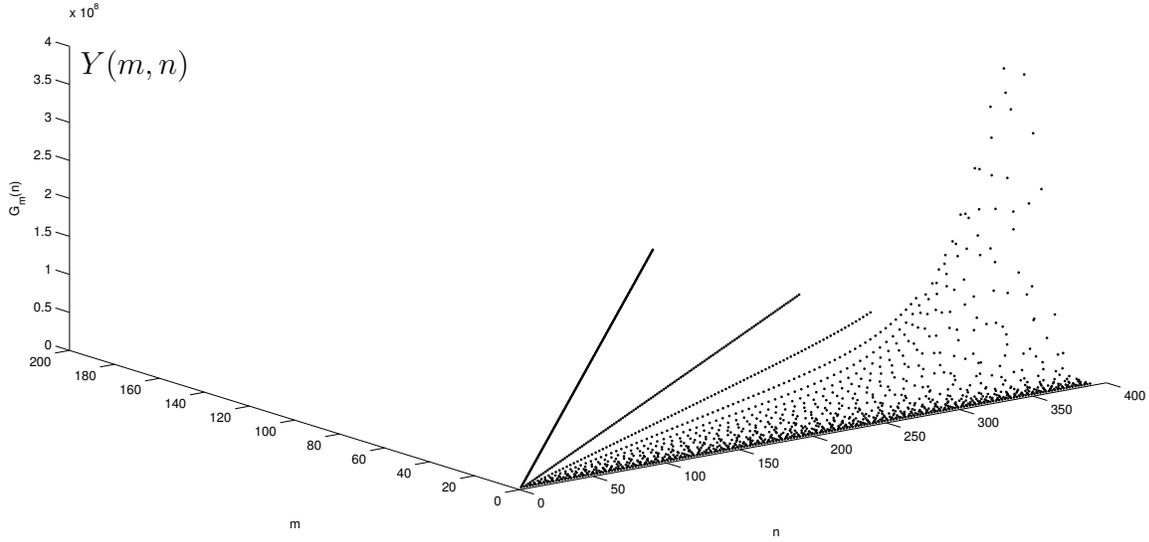
\vspace{-1.4in}

\begin{figure}[H]
\begin{minipage}{0.33\linewidth}
\centerline{\includegraphics[width=3.1in]{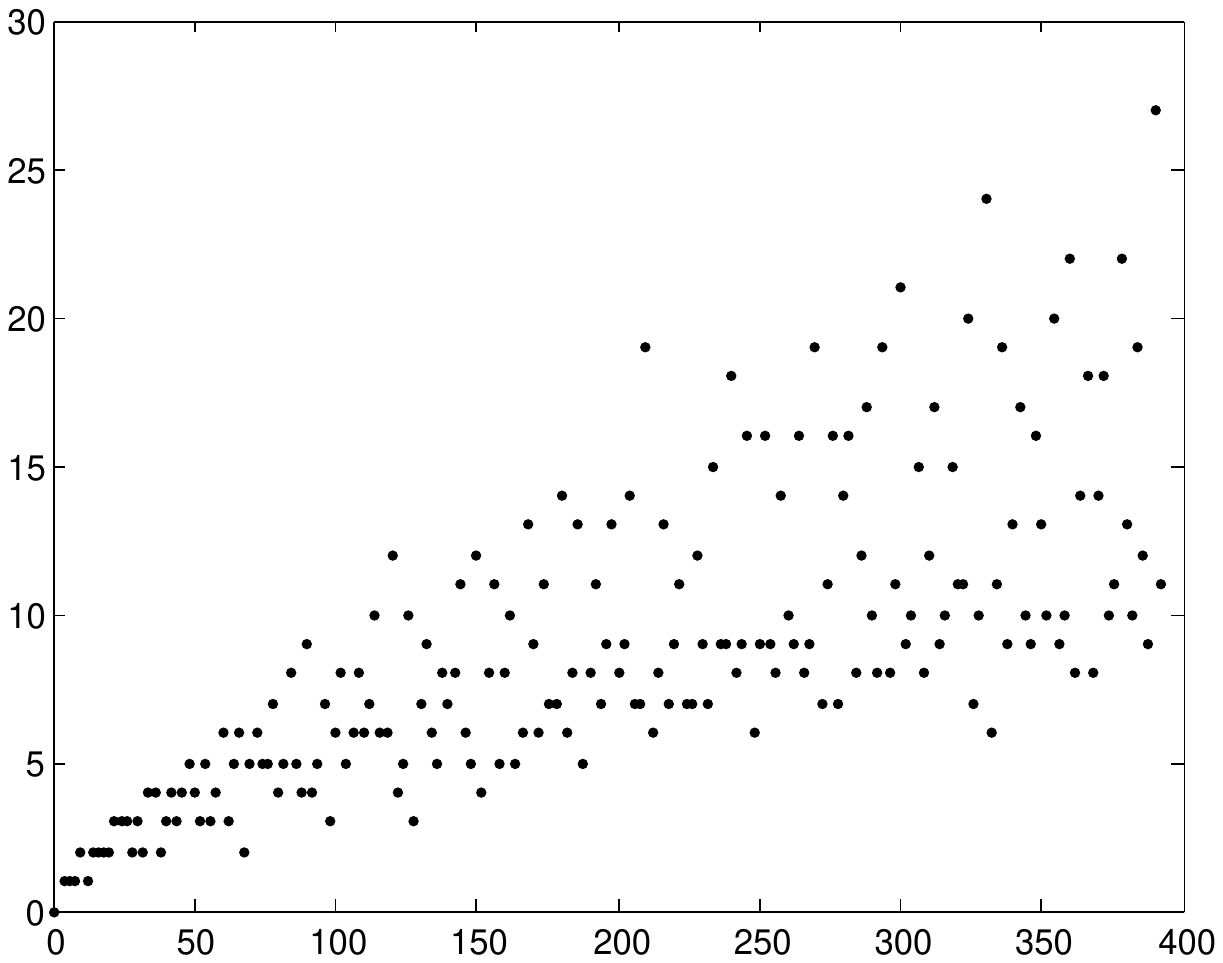}}
\end{minipage}
\hspace{-0.07in}
\begin{minipage}{0.33\linewidth}
\centerline{\includegraphics[width=3.1in]{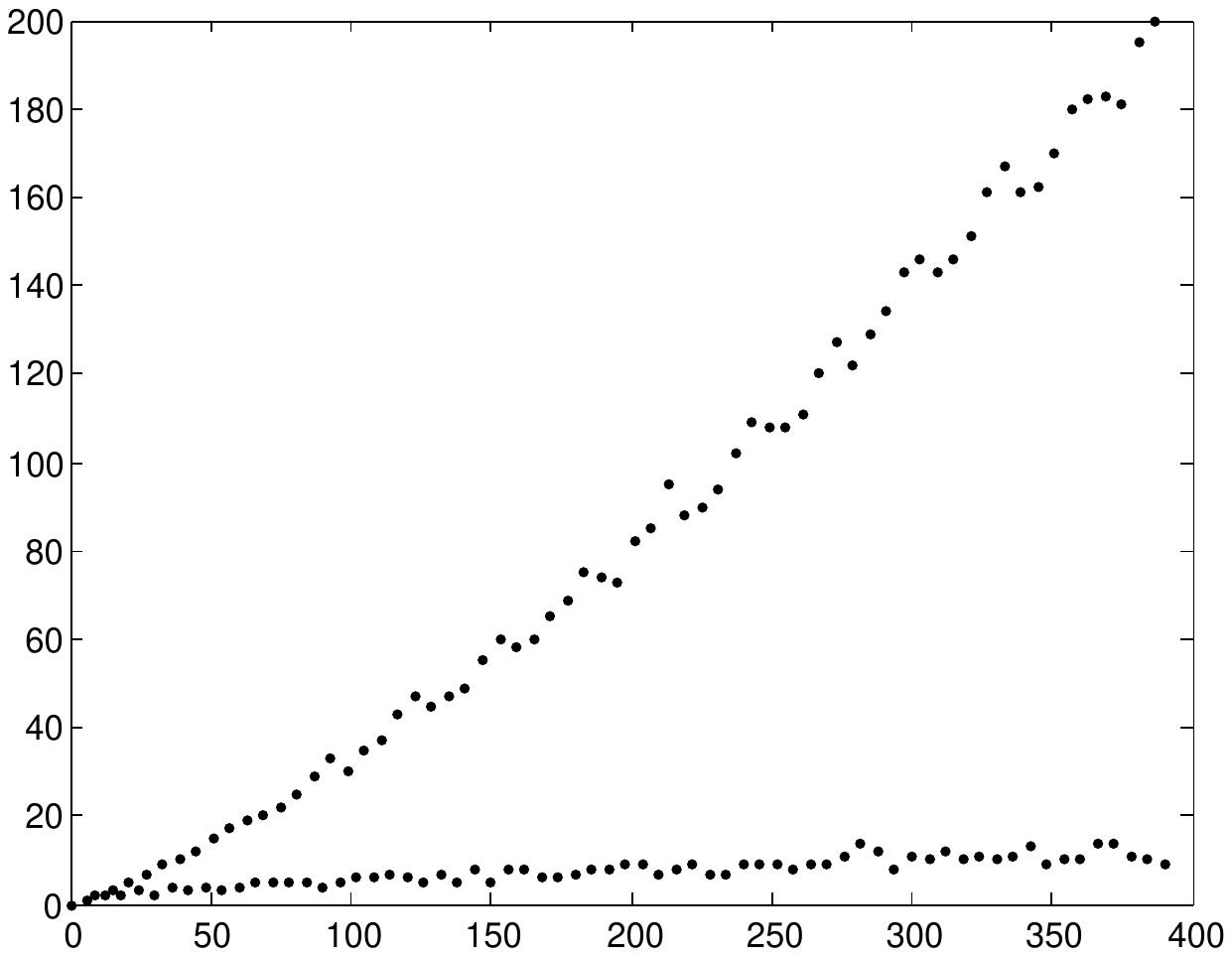}}
\end{minipage}
\hspace{-0.07in}
\begin{minipage}{0.33\linewidth}
\centerline{\includegraphics[width=3.1in]{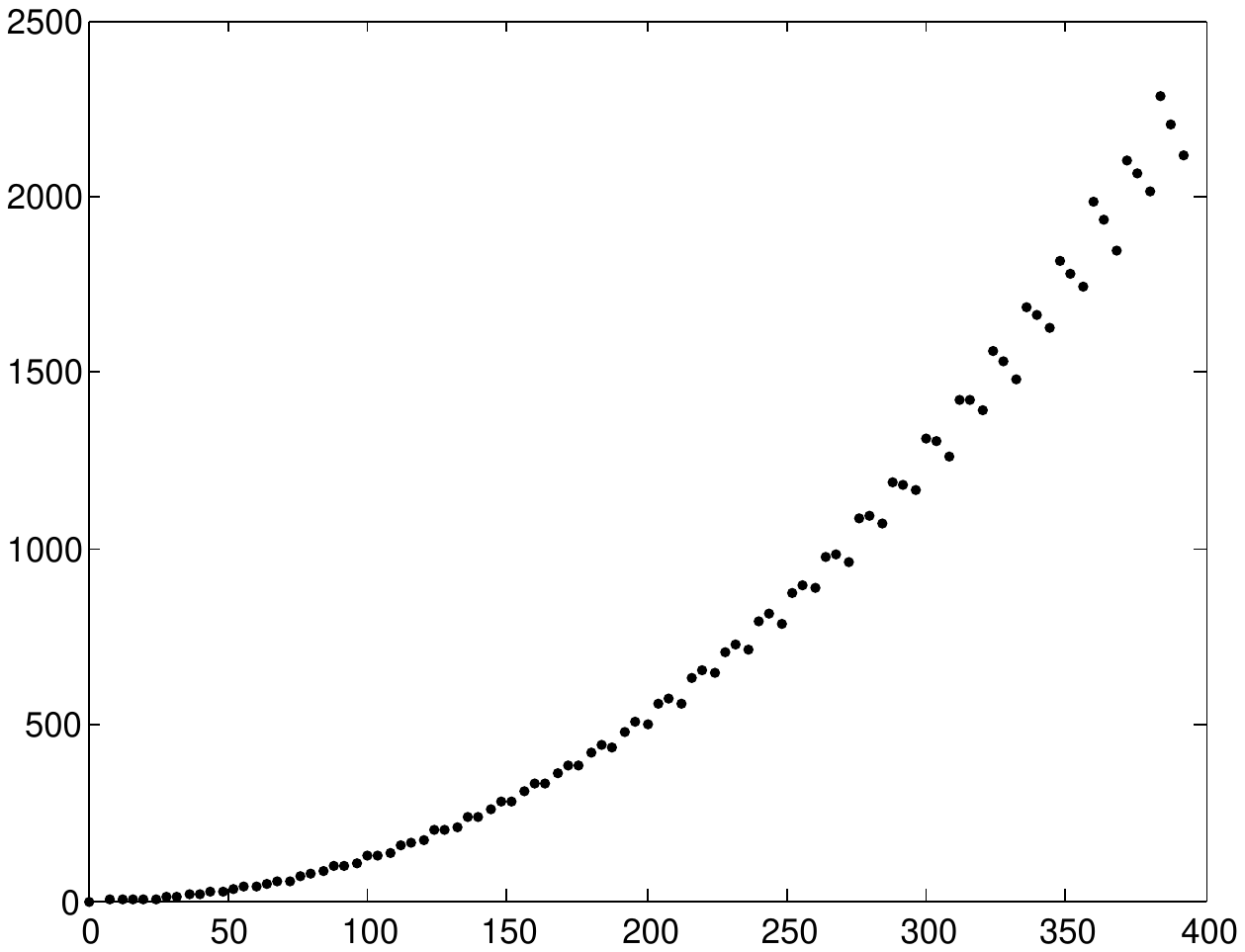}}
\end{minipage}
\vfill
\vspace{-2.7in}
\begin{minipage}{0.33\linewidth}
\centerline{\includegraphics[width=3.1in]{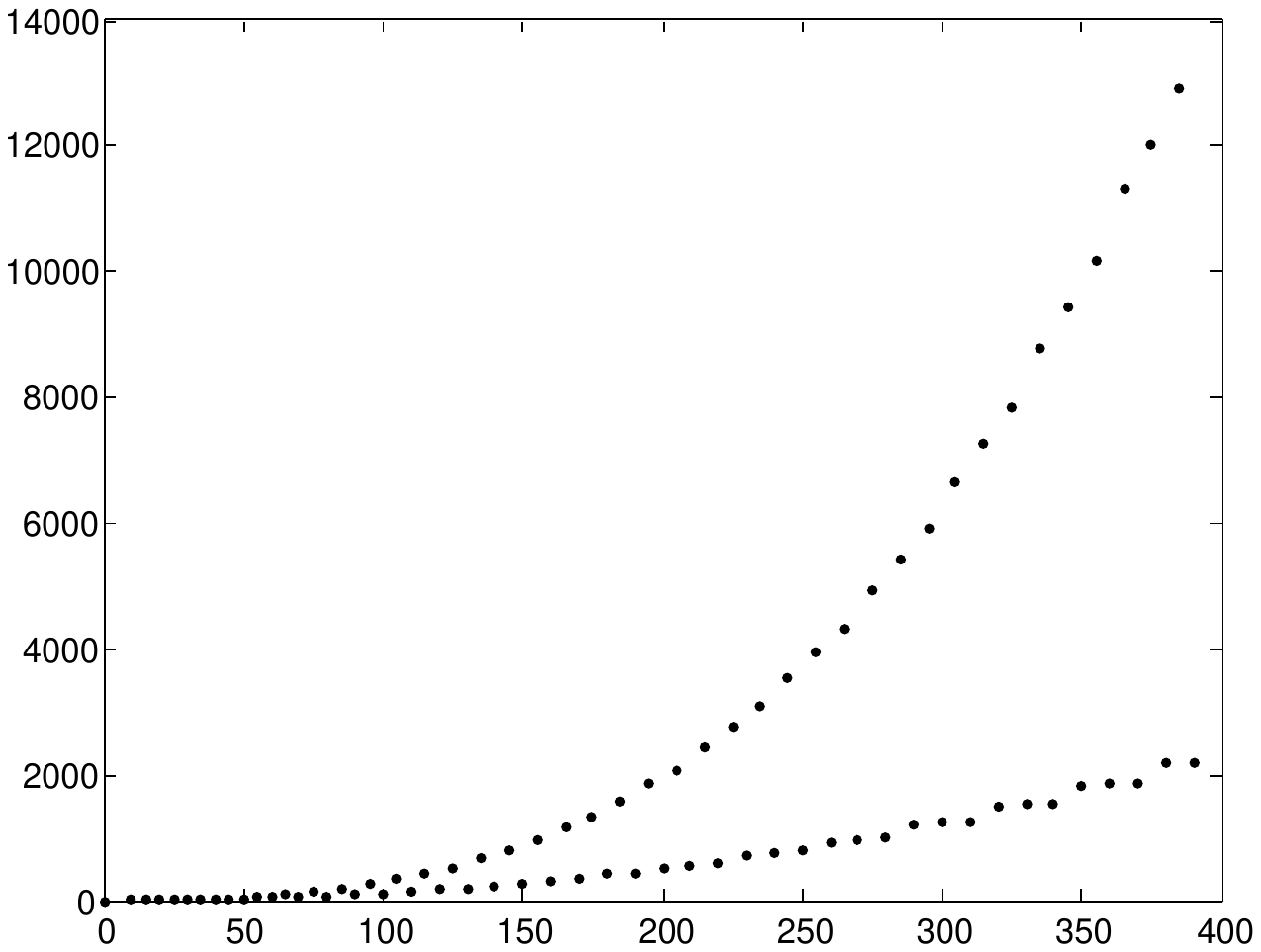}}
\end{minipage}
\hspace{-0.07in}
\begin{minipage}{0.33\linewidth}
\centerline{\includegraphics[width=3.1in]{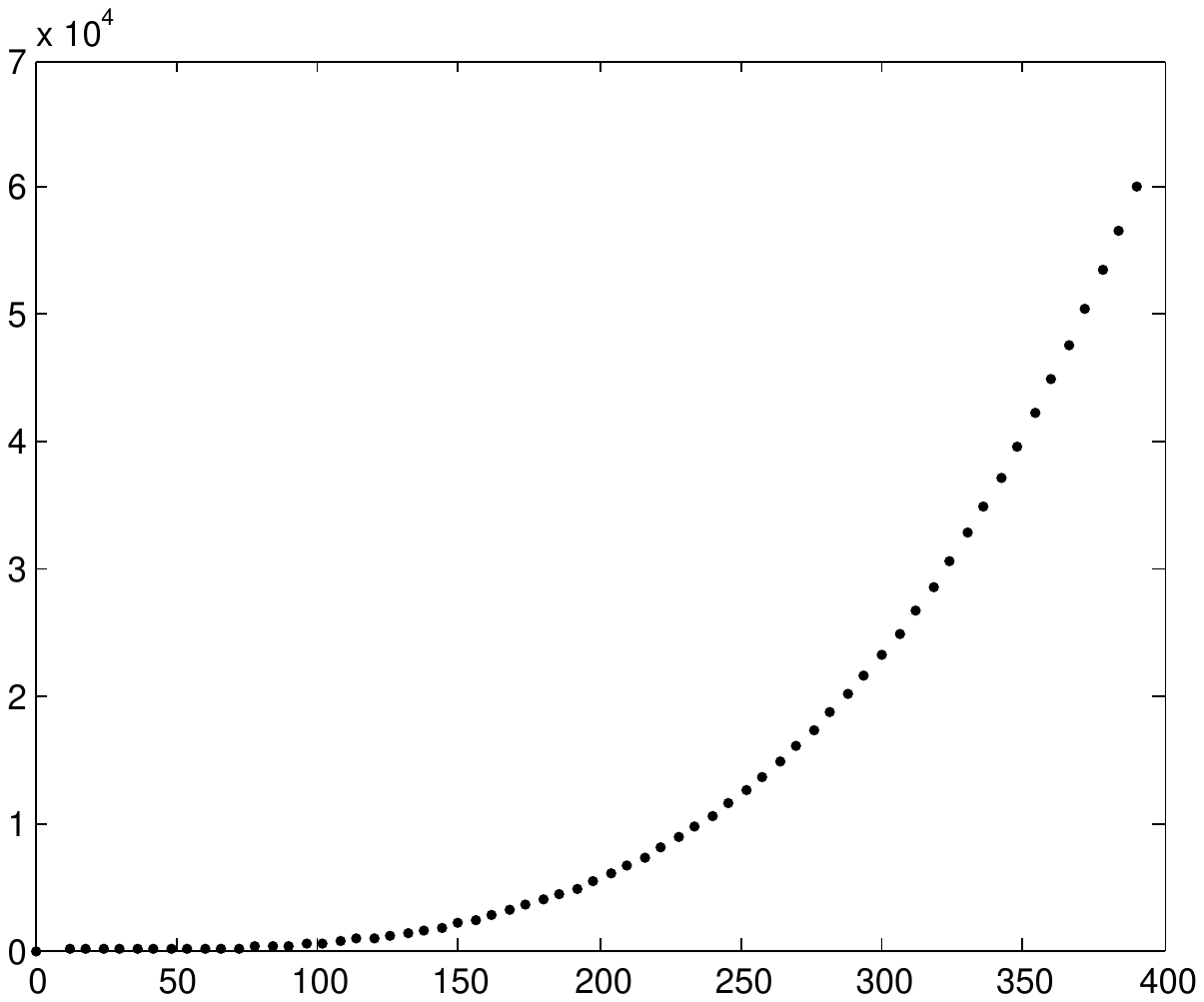}}
\end{minipage}
\hspace{-0.07in}
\begin{minipage}{0.33\linewidth}
\centerline{\includegraphics[width=3.1in]{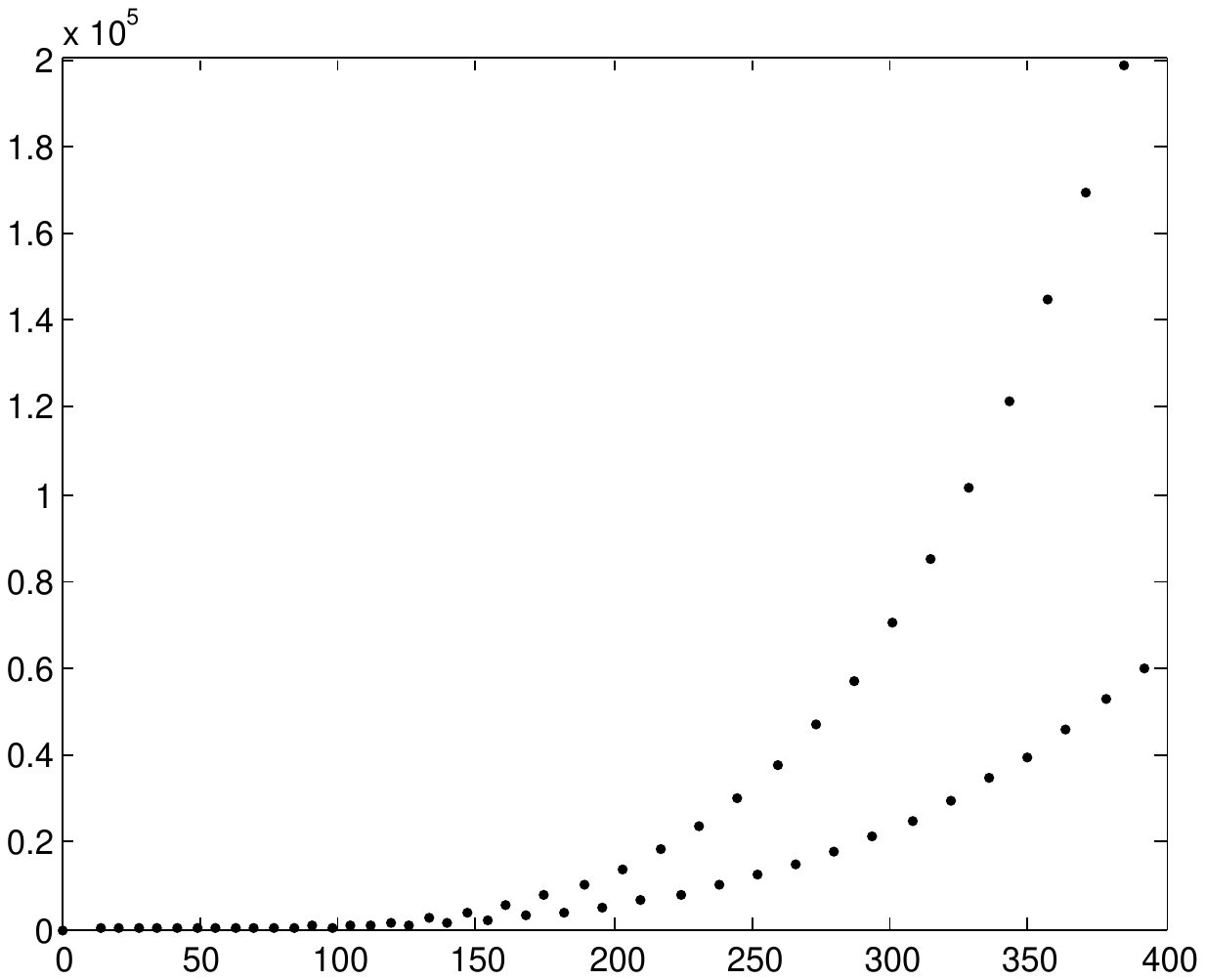}}
\end{minipage}
  \begin{picture}(0,0)
    \put(40,310){$m=2$}
    \put(180,310){$m=3$}
    \put(310,310){$m=4$}
    \put(40,190){$m=5$}
    \put(180,190){$m=6$}
    \put(310,190){$m=7$}
  \end{picture}
  \vspace{-1.3in}
\caption{\small The function $Y(m,n)$ when fix $m$: $n$ up to 400.}\label{diff_m}
\end{figure}

\begin{figure}
\begin{minipage}{0.33\linewidth}
\centerline{\includegraphics[width=3.1in]{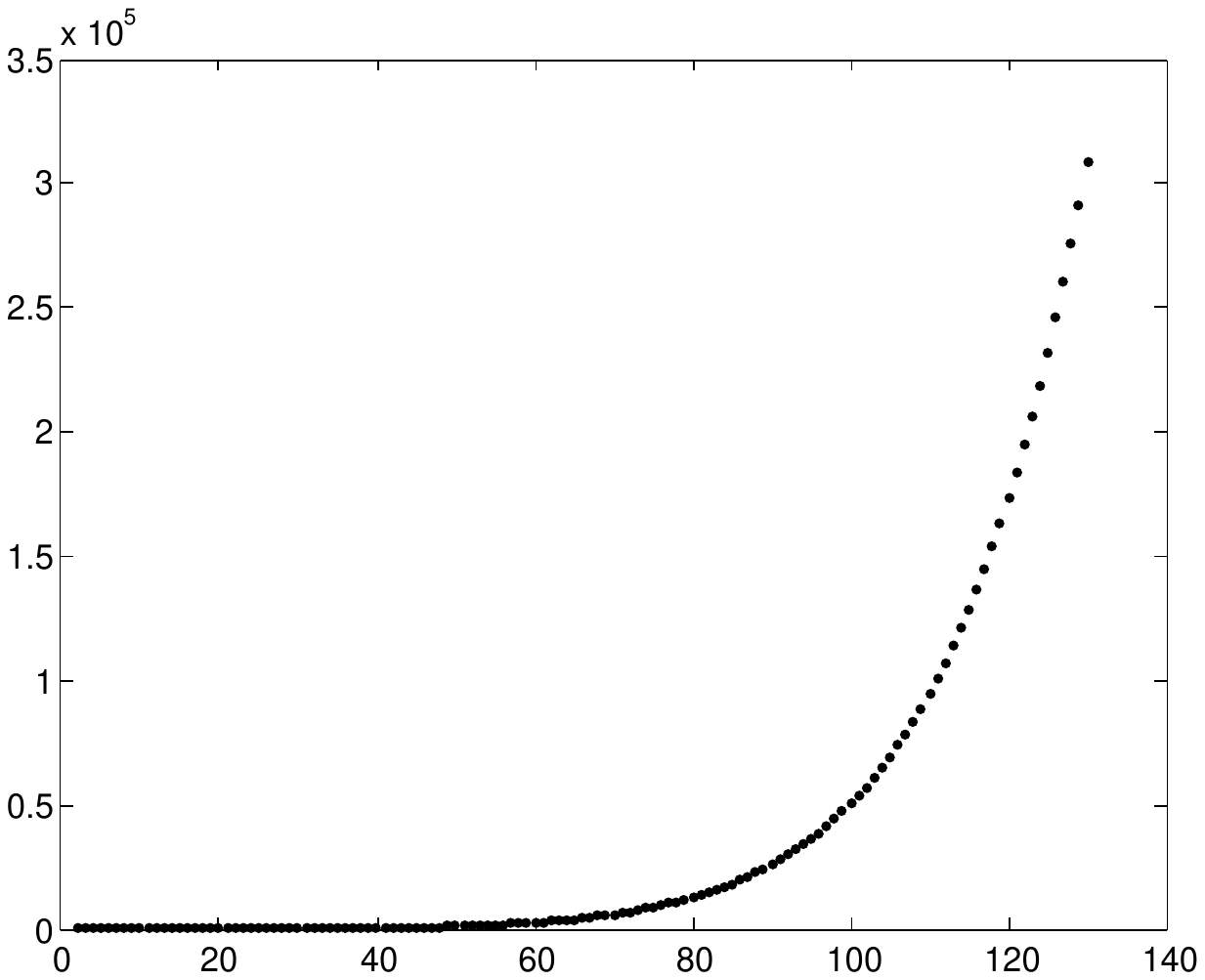}}
\end{minipage}
\hspace{-0.07in}
\begin{minipage}{0.33\linewidth}
\centerline{\includegraphics[width=3.1in]{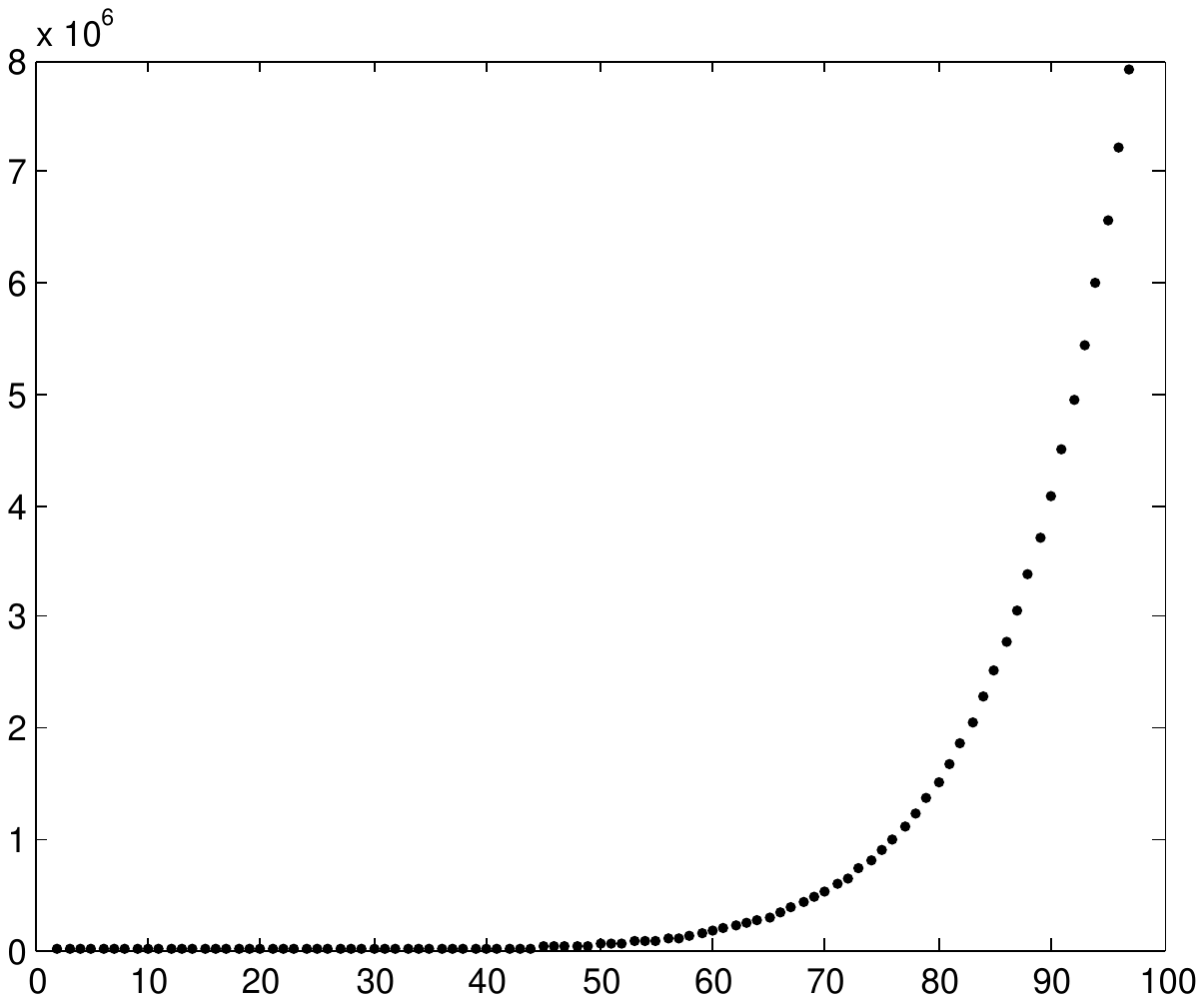}}
\end{minipage}
\hspace{-0.07in}
\begin{minipage}{0.33\linewidth}
\centerline{\includegraphics[width=3.1in]{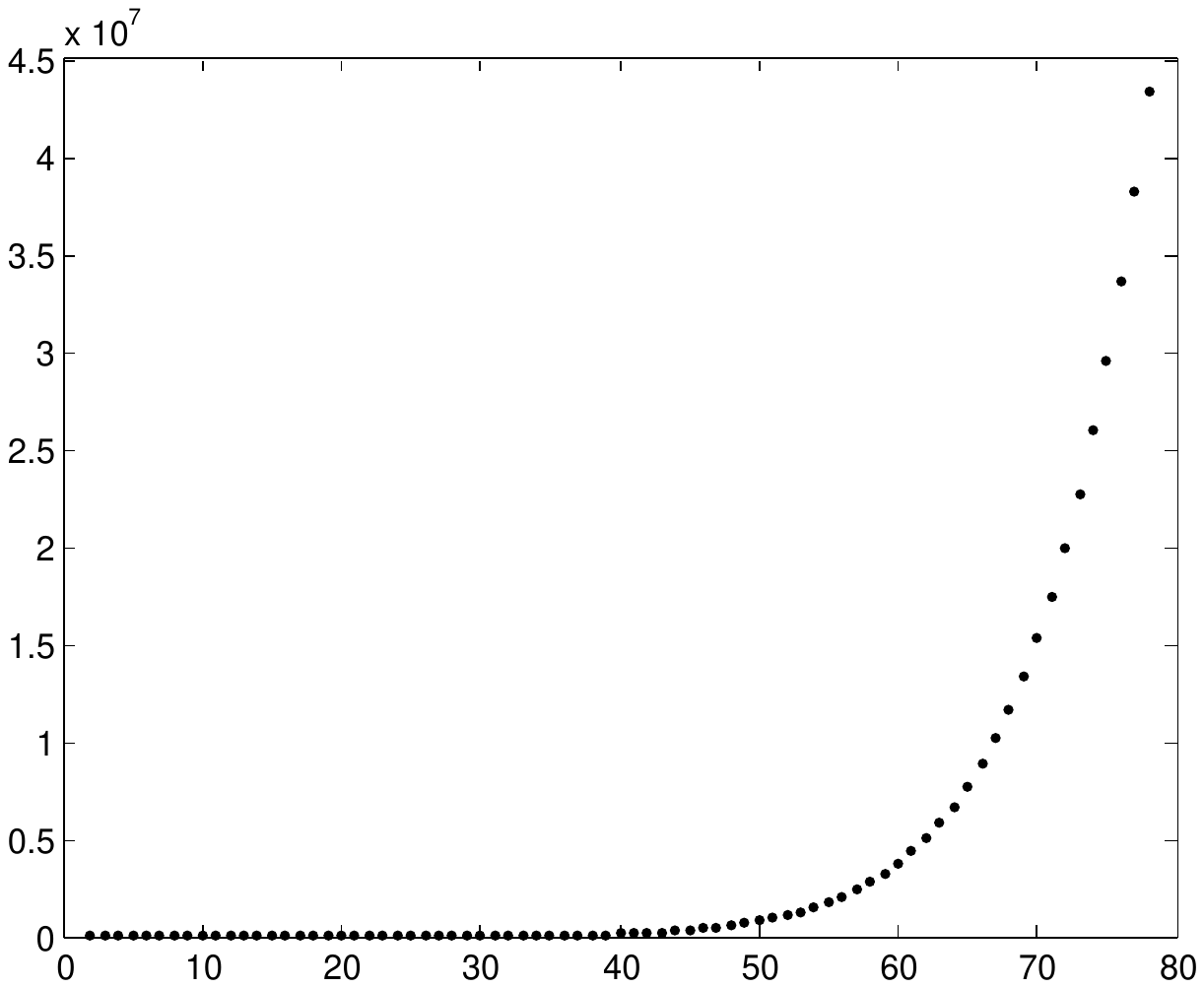}}
\end{minipage}
\vfill
\vspace{-2.7in}
\begin{minipage}{0.33\linewidth}
\centerline{\includegraphics[width=3.1in]{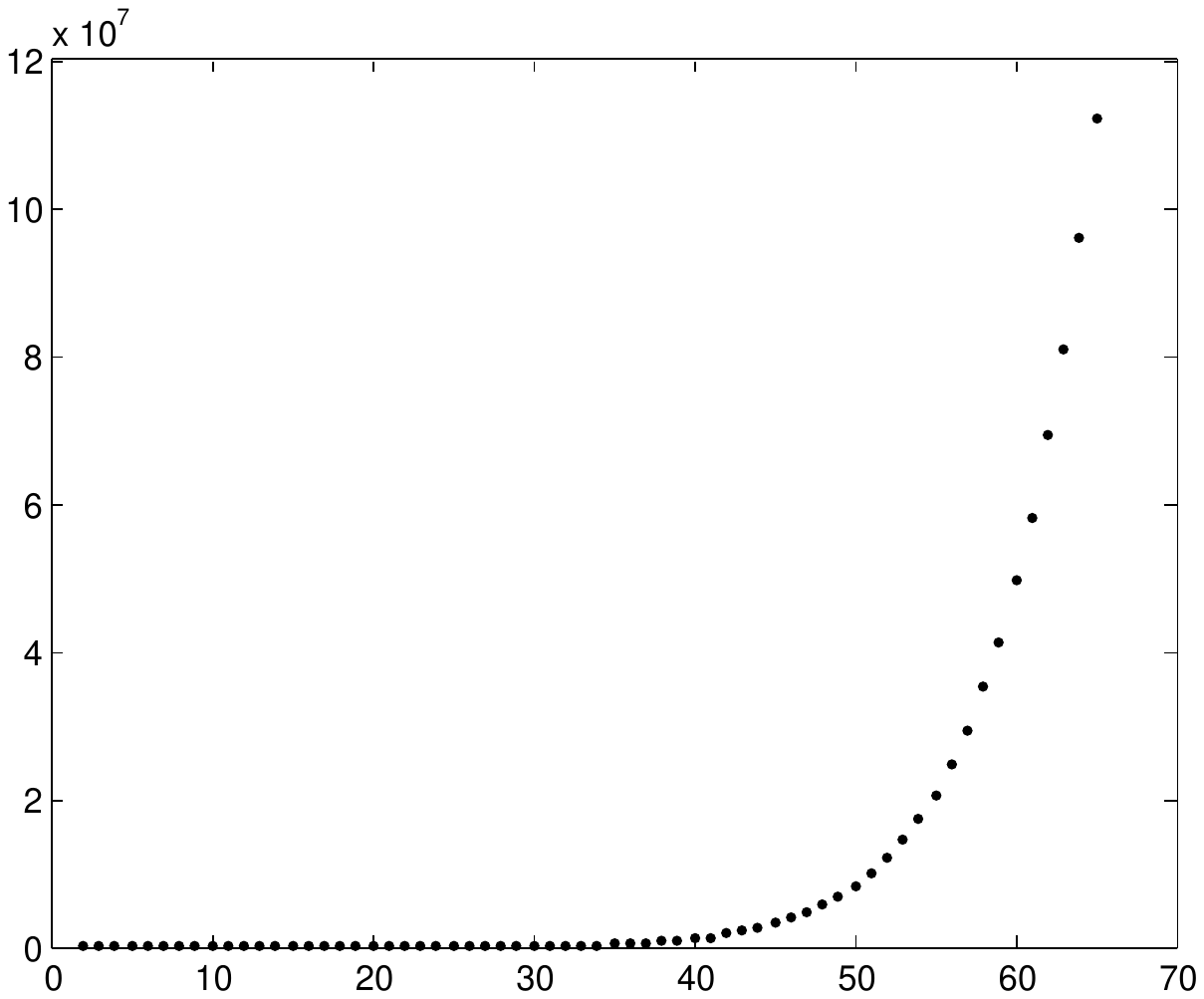}}
\end{minipage}
\hspace{-0.07in}
\begin{minipage}{0.33\linewidth}
\centerline{\includegraphics[width=3.1in]{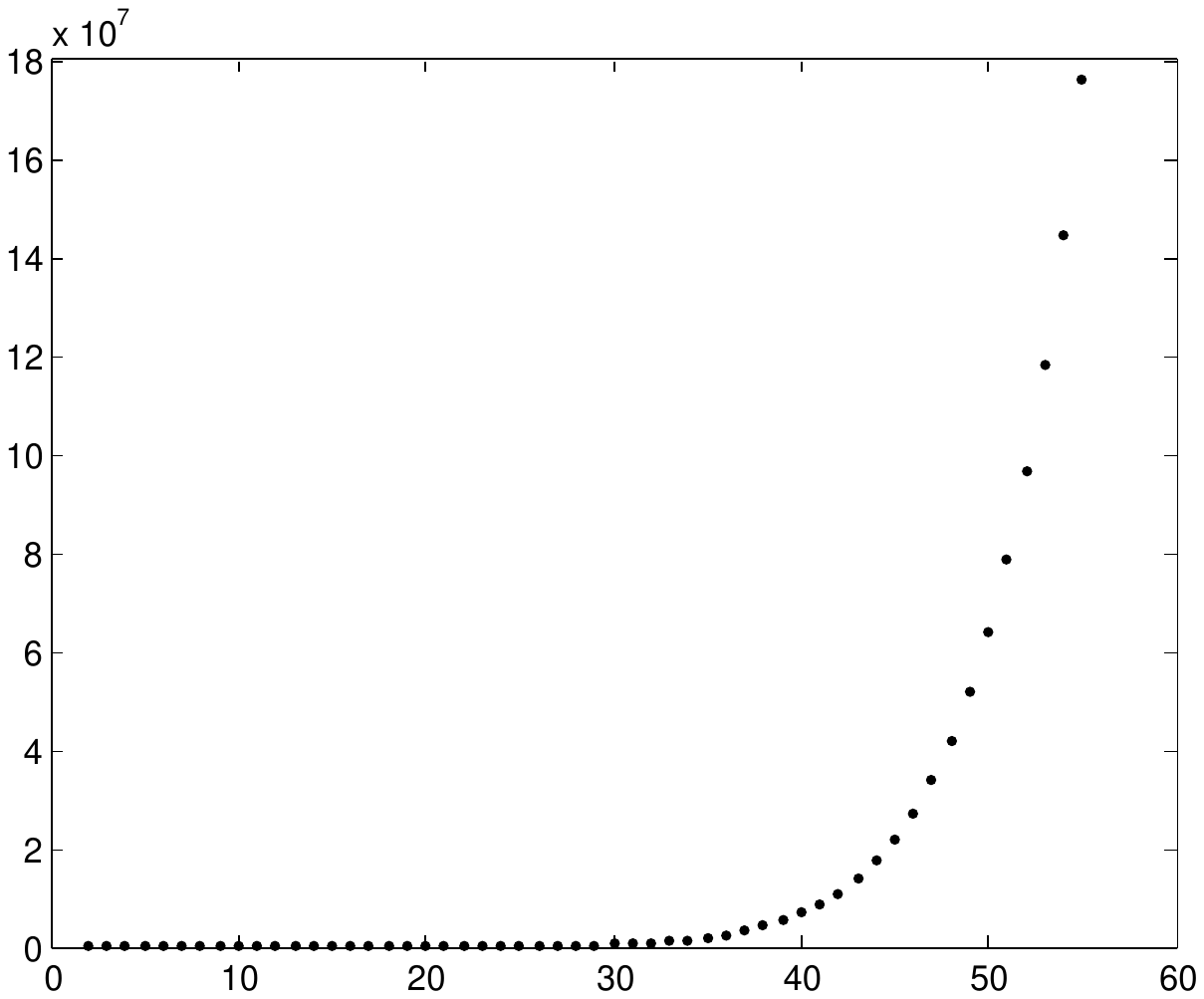}}
\end{minipage}
\hspace{-0.07in}
\begin{minipage}{0.33\linewidth}
\centerline{\includegraphics[width=3.1in]{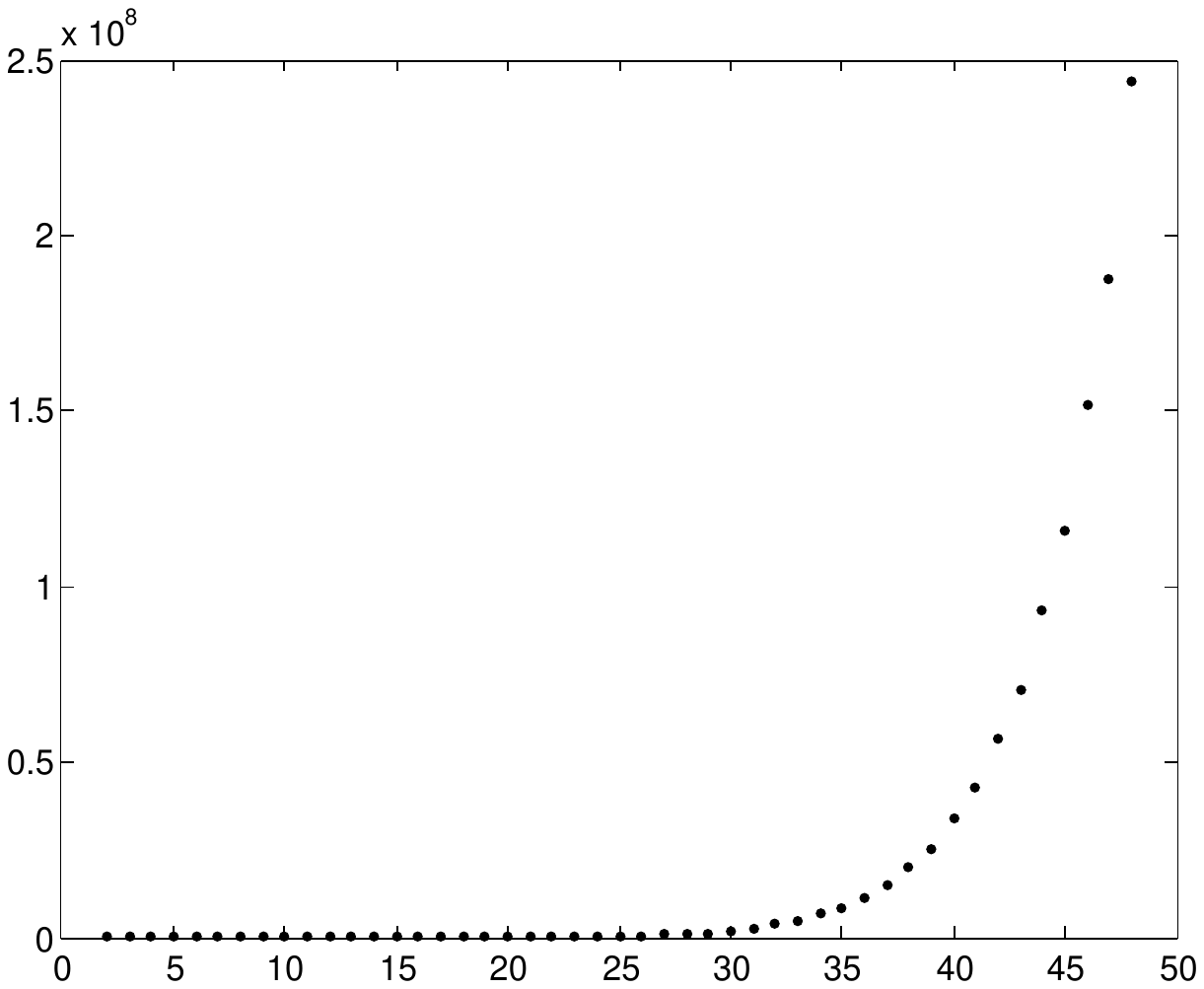}}
\end{minipage}
\vfill
\vspace{-2.7in}
\begin{minipage}{0.33\linewidth}
\centerline{\includegraphics[width=3.1in]{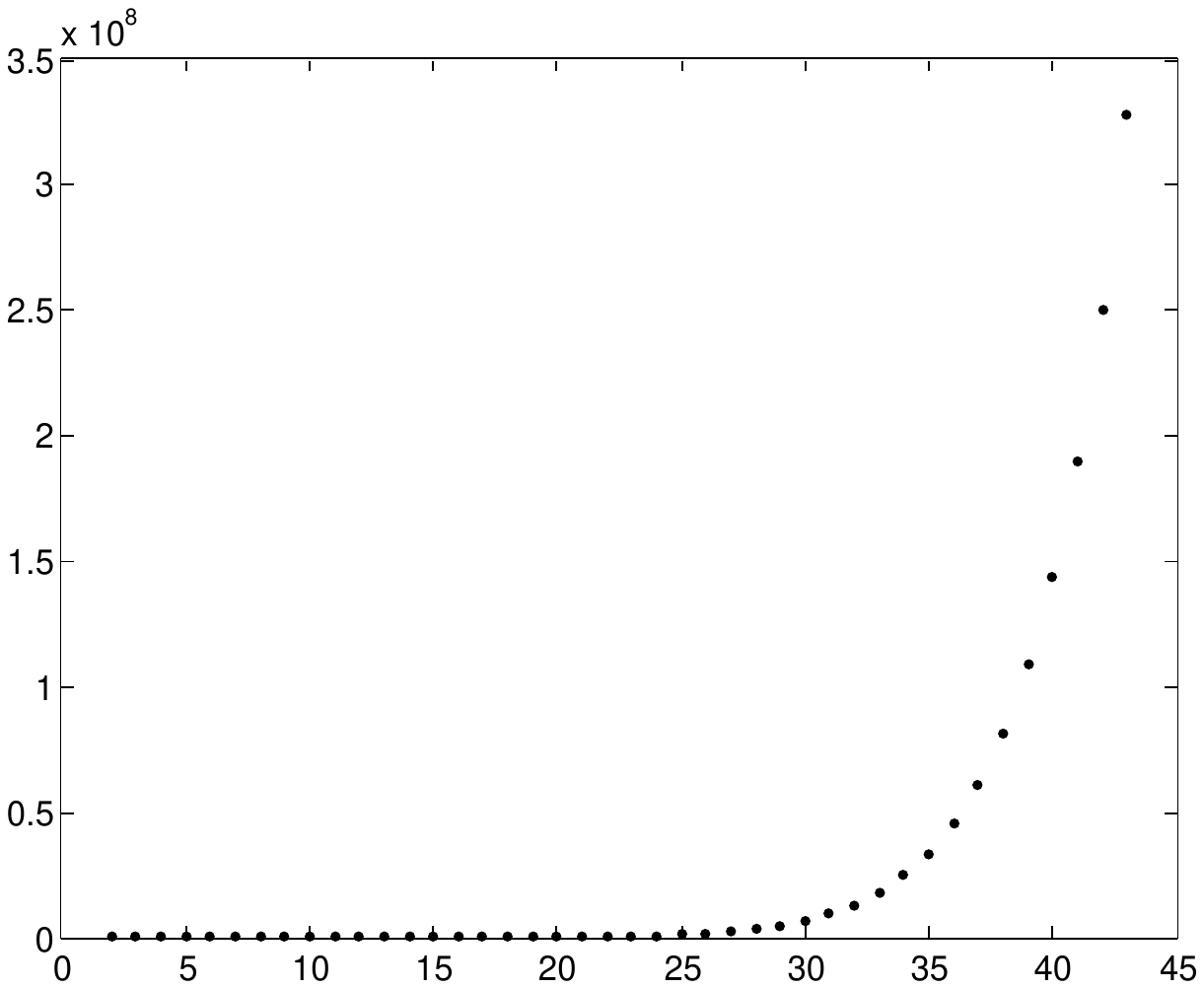}}
\end{minipage}
\hspace{-0.07in}
\begin{minipage}{0.33\linewidth}
\centerline{\includegraphics[width=3.1in]{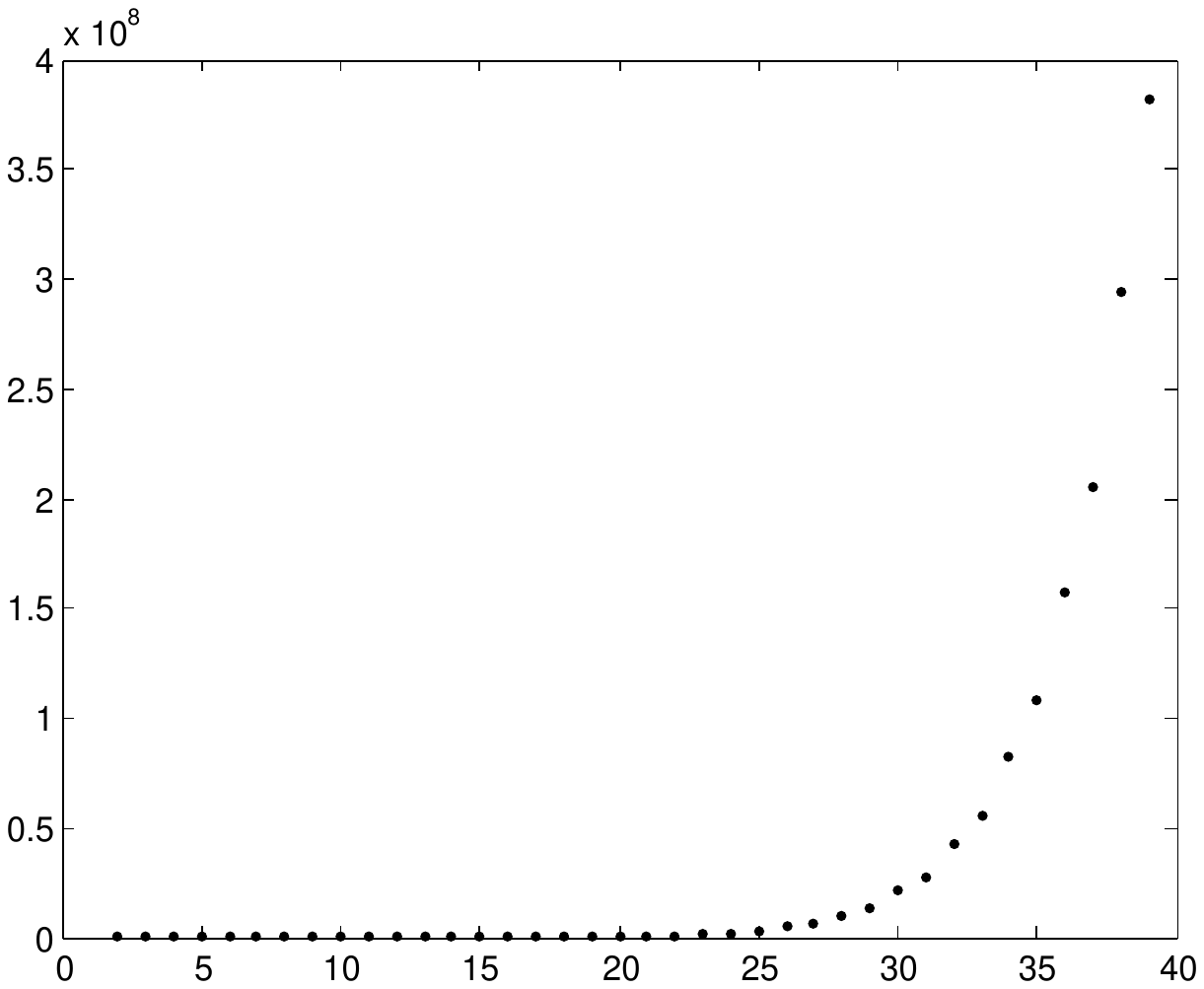}}
\end{minipage}
\hspace{-0.07in}
\begin{minipage}{0.33\linewidth}
\centerline{\includegraphics[width=3.1in]{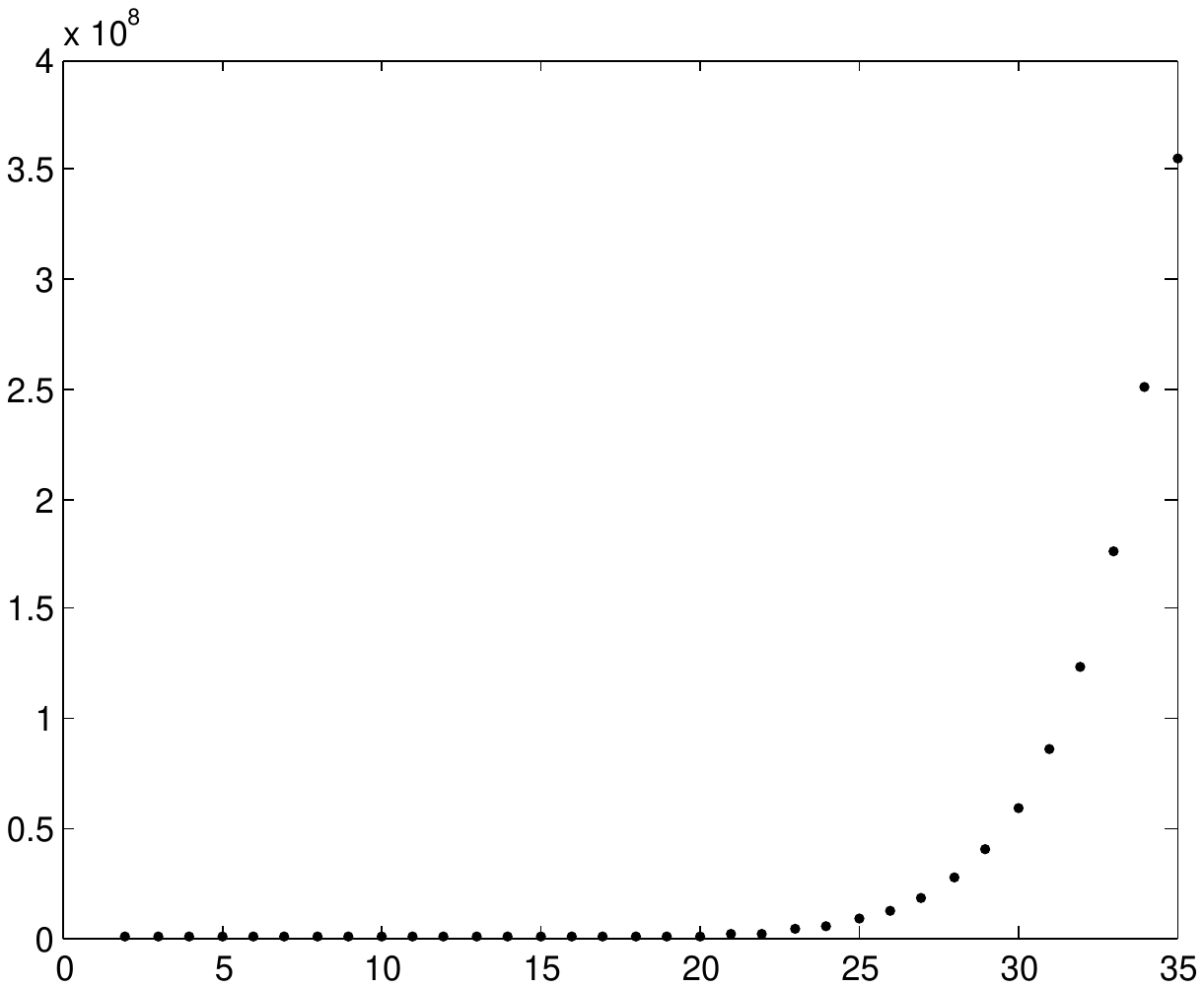}}
\end{minipage}
\begin{picture}(0,0)
    \put(40,430){$k=3$}
    \put(180,430){$k=4$}
    \put(310,430){$k=5$}
    \put(40,310){$k=6$}
    \put(180,310){$k=7$}
    \put(310,310){$k=8$}
    \put(40,190){$k=9$}
    \put(180,190){$k=10$}
    \put(310,190){$k=11$}
  \end{picture}
\vspace{-1.2in}
\caption{\small The function $Y(m,k\cdot m)$ with different $k$}\label{diff_k}
\end{figure}

Obviously, Fig. \ref{diff_k} is more regularity than in Fig. \ref{diff_m}. For the fixed $n/m$ quotient $k$, function $Y(m, km)$ has more obvious rule, which will be researched in the future. And by analyzing the shape of plots in Fig. \ref{diff_k}, the function $Y(m,n)$ is given asymptotically by
\begin{equation}
Y(m,n) \sim C\frac{1}{\pi(m)}e^{c(k)m},
\end{equation}
where $k=n/m$, $c(k)$ is a function with $c(2)=0$, $C$ is a constant. Appendix \uppercase\expandafter{\romannumeral 3} lists some empirical verification results of $Y(m,n)$.

\clearpage

\end{document}